\documentclass[12pt]{amsart}

\usepackage{amsthm}
\usepackage{amsmath}
\usepackage{amssymb}
\usepackage{latexsym}
\usepackage{verbatim}
\usepackage[mathscr]{eucal}
\usepackage{amscd}
\usepackage{bm}

\begin{document}

%Greek letters

\newcommand{\alp}{\alpha}
\newcommand{\bet}{\beta}
\newcommand{\gam}{\gamma}
\newcommand{\del}{\delta}
\newcommand{\eps}{\epsilon}
\newcommand{\zet}{\zeta}
\newcommand{\tht}{\theta}
\newcommand{\iot}{\iota}
\newcommand{\kap}{\kappa}
\newcommand{\lam}{\lambda}
\newcommand{\sig}{\sigma}
\newcommand{\ups}{\upsilon}
\newcommand{\ome}{\omega}
\newcommand{\vep}{\varepsilon}
\newcommand{\vth}{\vartheta}
\newcommand{\vpi}{\varpi}
\newcommand{\vrh}{\varrho}
\newcommand{\vsi}{\varsigma}
\newcommand{\vph}{\varphi}
\newcommand{\Gam}{\Gamma}
\newcommand{\Del}{\Delta}
\newcommand{\Tht}{\Theta}
\newcommand{\Lam}{\Lambda}
\newcommand{\Sig}{\Sigma}
\newcommand{\Ups}{\Upsilon}
\newcommand{\Ome}{\Omega}

%fraktur letters

\newcommand{\frka}{{\mathfrak a}}    \newcommand{\frkA}{{\mathfrak A}}
\newcommand{\frkb}{{\mathfrak b}}    \newcommand{\frkB}{{\mathfrak B}}
\newcommand{\frkc}{{\mathfrak c}}    \newcommand{\frkC}{{\mathfrak C}}
\newcommand{\frkd}{{\mathfrak d}}    \newcommand{\frkD}{{\mathfrak D}}
\newcommand{\frke}{{\mathfrak e}}    \newcommand{\frkE}{{\mathfrak E}}
\newcommand{\frkf}{{\mathfrak f}}    \newcommand{\frkF}{{\mathfrak F}}
\newcommand{\frkg}{{\mathfrak g}}    \newcommand{\frkG}{{\mathfrak G}}
\newcommand{\frkh}{{\mathfrak h}}    \newcommand{\frkH}{{\mathfrak H}}
\newcommand{\frki}{{\mathfrak i}}    \newcommand{\frkI}{{\mathfrak I}}
\newcommand{\frkj}{{\mathfrak j}}    \newcommand{\frkJ}{{\mathfrak J}}
\newcommand{\frkk}{{\mathfrak k}}    \newcommand{\frkK}{{\mathfrak K}}
\newcommand{\frkl}{{\mathfrak l}}    \newcommand{\frkL}{{\mathfrak L}}
\newcommand{\frkm}{{\mathfrak m}}    \newcommand{\frkM}{{\mathfrak M}}
\newcommand{\frkn}{{\mathfrak n}}    \newcommand{\frkN}{{\mathfrak N}}
\newcommand{\frko}{{\mathfrak o}}    \newcommand{\frkO}{{\mathfrak O}}
\newcommand{\frkp}{{\mathfrak p}}    \newcommand{\frkP}{{\mathfrak P}}
\newcommand{\frkq}{{\mathfrak q}}    \newcommand{\frkQ}{{\mathfrak Q}}
\newcommand{\frkr}{{\mathfrak r}}    \newcommand{\frkR}{{\mathfrak R}}
\newcommand{\frks}{{\mathfrak s}}    \newcommand{\frkS}{{\mathfrak S}}
\newcommand{\frkt}{{\mathfrak t}}    \newcommand{\frkT}{{\mathfrak T}}
\newcommand{\frku}{{\mathfrak u}}    \newcommand{\frkU}{{\mathfrak U}}
\newcommand{\frkv}{{\mathfrak v}}    \newcommand{\frkV}{{\mathfrak V}}
\newcommand{\frkw}{{\mathfrak w}}    \newcommand{\frkW}{{\mathfrak W}}
\newcommand{\frkx}{{\mathfrak x}}    \newcommand{\frkX}{{\mathfrak X}}
\newcommand{\frky}{{\mathfrak y}}    \newcommand{\frkY}{{\mathfrak Y}}
\newcommand{\frkz}{{\mathfrak z}}    \newcommand{\frkZ}{{\mathfrak Z}}

%math boldface latters

\newcommand{\bfa}{{\mathbf a}}    \newcommand{\bfA}{{\mathbf A}}
\newcommand{\bfb}{{\mathbf b}}    \newcommand{\bfB}{{\mathbf B}}
\newcommand{\bfc}{{\mathbf c}}    \newcommand{\bfC}{{\mathbf C}}
\newcommand{\bfd}{{\mathbf d}}    \newcommand{\bfD}{{\mathbf D}}
\newcommand{\bfe}{{\mathbf e}}    \newcommand{\bfE}{{\mathbf E}}
\newcommand{\bff}{{\mathbf f}}    \newcommand{\bfF}{{\mathbf F}}
\newcommand{\bfg}{{\mathbf g}}    \newcommand{\bfG}{{\mathbf G}}
\newcommand{\bfh}{{\mathbf h}}    \newcommand{\bfH}{{\mathbf H}}
\newcommand{\bfi}{{\mathbf i}}    \newcommand{\bfI}{{\mathbf I}}
\newcommand{\bfj}{{\mathbf j}}    \newcommand{\bfJ}{{\mathbf J}}
\newcommand{\bfk}{{\mathbf k}}    \newcommand{\bfK}{{\mathbf K}}
\newcommand{\bfl}{{\mathbf l}}    \newcommand{\bfL}{{\mathbf L}}
\newcommand{\bfm}{{\mathbf m}}    \newcommand{\bfM}{{\mathbf M}}
\newcommand{\bfn}{{\mathbf n}}    \newcommand{\bfN}{{\mathbf N}}
\newcommand{\bfo}{{\mathbf o}}    \newcommand{\bfO}{{\mathbf O}}
\newcommand{\bfp}{{\mathbf p}}    \newcommand{\bfP}{{\mathbf P}}
\newcommand{\bfq}{{\mathbf q}}    \newcommand{\bfQ}{{\mathbf Q}}
\newcommand{\bfr}{{\mathbf r}}    \newcommand{\bfR}{{\mathbf R}}
\newcommand{\bfs}{{\mathbf s}}    \newcommand{\bfS}{{\mathbf S}}
\newcommand{\bft}{{\mathbf t}}    \newcommand{\bfT}{{\mathbf T}}
\newcommand{\bfu}{{\mathbf u}}    \newcommand{\bfU}{{\mathbf U}}
\newcommand{\bfv}{{\mathbf v}}    \newcommand{\bfV}{{\mathbf V}}
\newcommand{\bfw}{{\mathbf w}}    \newcommand{\bfW}{{\mathbf W}}
\newcommand{\bfx}{{\mathbf x}}    \newcommand{\bfX}{{\mathbf X}}
\newcommand{\bfy}{{\mathbf y}}    \newcommand{\bfY}{{\mathbf Y}}
\newcommand{\bfz}{{\mathbf z}}    \newcommand{\bfZ}{{\mathbf Z}}

%caligraphic letters

%\newcommand{\cal}{\fam2}
\newcommand{\cala}{{\mathcal A}}
\newcommand{\calb}{{\mathcal B}}
\newcommand{\calc}{{\mathcal C}}
\newcommand{\cald}{{\mathcal D}}
\newcommand{\cale}{{\mathcal E}}
\newcommand{\calf}{{\mathcal F}}
\newcommand{\calg}{{\mathcal G}}
\newcommand{\calh}{{\mathcal H}}
\newcommand{\cali}{{\mathcal I}}
\newcommand{\calj}{{\mathcal J}}
\newcommand{\calk}{{\mathcal K}}
\newcommand{\call}{{\mathcal L}}
\newcommand{\calm}{{\mathcal M}}
\newcommand{\caln}{{\mathcal N}}
\newcommand{\calo}{{\mathcal O}}
\newcommand{\calp}{{\mathcal P}}
\newcommand{\calq}{{\mathcal Q}}
\newcommand{\calr}{{\mathcal R}}
\newcommand{\cals}{{\mathcal S}}
\newcommand{\calt}{{\mathcal T}}
\newcommand{\calu}{{\mathcal U}}
\newcommand{\calv}{{\mathcal V}}
\newcommand{\calw}{{\mathcal W}}
\newcommand{\calx}{{\mathcal X}}
\newcommand{\caly}{{\mathcal Y}}
\newcommand{\calz}{{\mathcal Z}}

%math script

\newcommand{\scra}{{\mathscr A}}
\newcommand{\scrb}{{\mathscr B}}
\newcommand{\scrc}{{\mathscr C}}
\newcommand{\scrd}{{\mathscr D}}
\newcommand{\scre}{{\mathscr E}}
\newcommand{\scrf}{{\mathscr F}}
\newcommand{\scrg}{{\mathscr G}}
\newcommand{\scrh}{{\mathscr H}}
\newcommand{\scri}{{\mathscr I}}
\newcommand{\scrj}{{\mathscr J}}
\newcommand{\scrk}{{\mathscr K}}
\newcommand{\scrl}{{\mathscr L}}
\newcommand{\scrm}{{\mathscr M}}
\newcommand{\scrn}{{\mathscr N}}
\newcommand{\scro}{{\mathscr O}}
\newcommand{\scrp}{{\mathscr P}}
\newcommand{\scrq}{{\mathscr Q}}
\newcommand{\scrr}{{\mathscr R}}
\newcommand{\scrs}{{\mathscr S}}
\newcommand{\scrt}{{\mathscr T}}
\newcommand{\scru}{{\mathscr U}}
\newcommand{\scrv}{{\mathscr V}}
\newcommand{\scrw}{{\mathscr W}}
\newcommand{\scrx}{{\mathscr X}}
\newcommand{\scry}{{\mathscr Y}}
\newcommand{\scrz}{{\mathscr Z}}

%math Bbb

\newcommand{\AAA}{{\mathbb A}} %not \AA
\newcommand{\BB}{{\mathbb B}}
\newcommand{\CC}{{\mathbb C}}
\newcommand{\DD}{{\mathbb D}}
\newcommand{\EE}{{\mathbb E}}
\newcommand{\FF}{{\mathbb F}}
\newcommand{\GG}{{\mathbb G}}
\newcommand{\HH}{{\mathbb H}}
\newcommand{\II}{{\mathbb I}}
\newcommand{\JJ}{{\mathbb J}}
\newcommand{\KK}{{\mathbb K}}
\newcommand{\LL}{{\mathbb L}}
\newcommand{\MM}{{\mathbb M}}
\newcommand{\NN}{{\mathbb N}}
\newcommand{\OO}{{\mathbb O}}
\newcommand{\PP}{{\mathbb P}}
\newcommand{\QQ}{{\mathbb Q}}
\newcommand{\RR}{{\mathbb R}}
\newcommand{\SSS}{{\mathbb S}} %not \SS
\newcommand{\TT}{{\mathbb T}}
\newcommand{\UU}{{\mathbb U}}
\newcommand{\VV}{{\mathbb V}}
\newcommand{\WW}{{\mathbb W}}
\newcommand{\XX}{{\mathbb X}}
\newcommand{\YY}{{\mathbb Y}}
\newcommand{\ZZ}{{\mathbb Z}}

%typewriter

\newcommand{\tta}{\hbox{\tt a}}    \newcommand{\ttA}{\hbox{\tt A}}
\newcommand{\ttb}{\hbox{\tt b}}    \newcommand{\ttB}{\hbox{\tt B}}
\newcommand{\ttc}{\hbox{\tt c}}    \newcommand{\ttC}{\hbox{\tt C}}
\newcommand{\ttd}{\hbox{\tt d}}    \newcommand{\ttD}{\hbox{\tt D}}
\newcommand{\tte}{\hbox{\tt e}}    \newcommand{\ttE}{\hbox{\tt E}}
\newcommand{\ttf}{\hbox{\tt f}}    \newcommand{\ttF}{\hbox{\tt F}}
\newcommand{\ttg}{\hbox{\tt g}}    \newcommand{\ttG}{\hbox{\tt G}}
\newcommand{\tth}{\hbox{\tt h}}    \newcommand{\ttH}{\hbox{\tt H}}
\newcommand{\tti}{\hbox{\tt i}}    \newcommand{\ttI}{\hbox{\tt I}}
\newcommand{\ttj}{\hbox{\tt j}}    \newcommand{\ttJ}{\hbox{\tt J}}
\newcommand{\ttk}{\hbox{\tt k}}    \newcommand{\ttK}{\hbox{\tt K}}
\newcommand{\ttl}{\hbox{\tt l}}    \newcommand{\ttL}{\hbox{\tt L}}
\newcommand{\ttm}{\hbox{\tt m}}    \newcommand{\ttM}{\hbox{\tt M}}
\newcommand{\ttn}{\hbox{\tt n}}    \newcommand{\ttN}{\hbox{\tt N}}
\newcommand{\tto}{\hbox{\tt o}}    \newcommand{\ttO}{\hbox{\tt O}}
\newcommand{\ttp}{\hbox{\tt p}}    \newcommand{\ttP}{\hbox{\tt P}}
\newcommand{\ttq}{\hbox{\tt q}}    \newcommand{\ttQ}{\hbox{\tt Q}}
\newcommand{\ttr}{\hbox{\tt r}}    \newcommand{\ttR}{\hbox{\tt R}}
\newcommand{\tts}{\hbox{\tt s}}    \newcommand{\ttS}{\hbox{\tt S}}
\newcommand{\ttt}{\hbox{\tt t}}    \newcommand{\ttT}{\hbox{\tt T}}
\newcommand{\ttu}{\hbox{\tt u}}    \newcommand{\ttU}{\hbox{\tt U}}
\newcommand{\ttv}{\hbox{\tt v}}    \newcommand{\ttV}{\hbox{\tt V}}
\newcommand{\ttw}{\hbox{\tt w}}    \newcommand{\ttW}{\hbox{\tt W}}
\newcommand{\ttx}{\hbox{\tt x}}    \newcommand{\ttX}{\hbox{\tt X}}
\newcommand{\tty}{\hbox{\tt y}}    \newcommand{\ttY}{\hbox{\tt Y}}
\newcommand{\ttz}{\hbox{\tt z}}    \newcommand{\ttZ}{\hbox{\tt Z}}

\newcommand{\phm}{\phantom}
\newcommand{\ds}{\displaystyle }
\newcommand{\smallstrut}{\vphantom{\vrule height 3pt }}
\def\bdm #1#2#3#4{\left(
\begin{array} {c|c}{\ds{#1}}
 & {\ds{#2}} \\ \hline
{\ds{#3}\vphantom{\ds{#3}^1}} &  {\ds{#4}}
\end{array}
\right)}
\newcommand{\wtd}{\widetilde }
\newcommand{\bsl}{\backslash }
\newcommand{\GL}{{\mathrm{GL}}}
\newcommand{\SL}{{\mathrm{SL}}}
\newcommand{\GSp}{{\mathrm{GSp}}}
\newcommand{\PGSp}{{\mathrm{PGSp}}}
\newcommand{\SP}{{\mathrm{Sp}}}
\newcommand{\SO}{{\mathrm{SO}}}
\newcommand{\SU}{{\mathrm{SU}}}
\newcommand{\Ind}{\mathrm{Ind}}
\newcommand{\Hom}{{\mathrm{Hom}}}
\newcommand{\Ad}{{\mathrm{Ad}}}
\newcommand{\Sym}{{\mathrm{Sym}}}
\newcommand{\Mat}{\mathrm{M}}
\newcommand{\sgn}{\mathrm{sgn}}
\newcommand{\trs}{\,^t\!}
\newcommand{\iu}{\sqrt{-1}}
\newcommand{\oo}{\hbox{\bf 0}}
\newcommand{\ono}{\hbox{\bf 1}}
\newcommand{\smallcirc}{\lower .3em \hbox{\rm\char'27}\!}
\newcommand{\bAf}{\bA_{\hbox{\eightrm f}}}
\newcommand{\thalf}{{\textstyle{\frac12}}}
\newcommand{\shp}{\hbox{\rm\char'43}}
\newcommand{\Gal}{\operatorname{Gal}}
\newcommand{\opeke}{\frko^\times}
\newcommand{\ktwo}{F/F^{\times 2}}
\newcommand{\tfourth}{\frac{1}{4}}
\newcommand{\bdel}{\bm\del}
\newcommand{\vsst}{V^{\mathrm{ss}}}
\theoremstyle{plain}
\newtheorem{theorem}{Theorem}[section]
\newtheorem{lemma}{Lemma}[section]
\newtheorem{proposition}{Proposition}[section]
\theoremstyle{definition}
\newtheorem{definition}{Definition}[section]
\newtheorem{conjecture}{Conjecture}[section]
\theoremstyle{remark}
\newtheorem{remark}{Remark}[section]
\newtheorem{corollary}{Corollary}
  \makeatletter
    \renewcommand{\theequation}{%
    \thesection.\arabic{equation}}
    \@addtoreset{equation}{section}
  \makeatother

%%%%%%%%%%%%%%%%title%%%%%%%%%%%%%%%%%%%
%
\title[]{On the functional equation of the Siegel series}
\author{Tamotsu Ikeda}
\address{Graduate school of mathematics, Kyoto University, Kitashirakawa, Kyoto, 606-8502, Japan}
\email{ikeda@math.kyoto-u.ac.jp}
%\date{}
%\thanks{}
\maketitle
\begin{abstract}
It is well-known that the Fourier coefficients of Siegel-Eisenstein series can be expressed in terms of the Siegel series.
The functional equation of the Siegel series of a quadratic form over $\mathbb{Q}_p$ was first proved by Katsurada.
In this paper, we prove the functional equation of the Siegel series over a non-archimedean local field by using the representation theoretic argument by Kudla and Sweet.
\end{abstract}
\centerline{\textbf{Introduction}}
\bigskip

The theory of Siegel series was initiated by Siegel \cite{siegel35} to investigate the Fourier coefficients of the Siegel Eisenstein series.
Since then, many authors treated Siegel series.
Katsurada \cite{katsurada} gave an explicit formula for the Siegel series over $\QQ_p$.
To obtain the explicit formula, Katsurada proved a functional equation of the Siegel series, which is now called the Katsurada functional equation.
The purpose of this paper is to generalize Katsurada functional equation over an arbitrary local field of characteristic not $2$.

There are several proofs of the Katsurada functional equation over $\QQ_p$.
B\"ocherer and Kohnen \cite{boechererkohnen} used the global functional equation of the Siegel Eisenstein series.
The proof of Sato and Hironaka \cite{Satohironaka} used the theory of spherical functions.
In fact, Karel \cite{karel} has shown that there exists a functional equation by using the representation theory, but he did not calculate a precise form of the functional equation.
The precise form of the functional equation can be calculated by using the result of Sweet \cite{sweet} on the ``gamma matrix'' of a prehomogeneous vector space, in principle.

In this paper, we first reformulated the result of Sweet \cite{sweet} suitable for our purpose.
Let $\mathrm{Sym}_n(F)$ be the space of symmetric matrices over a non-archimedean local field $F$ of characteristic not $2$.
We will calculate a precise form of the local functional equation of the prehomogeneous vector space $\mathrm{Sym}_n(F)$.
Our method of the calculation is basically the same as that of Sato \cite{fumihiro}.

We now explain the content of this paper.
In section \ref{sec:1}, we give a preliminary result on the Weil constants and Tate's local factors.
In section \ref{sec:2}, we give a local functional equation (Theorem \ref{thm:2.1} and Theorem \ref{thm:2.2}) for the prehomogeneous vector space $\mathrm{Sym}_n(F)$.
In these theorems, we consider the zeta integrals with respect to a character $\ome$ of $F^\times$.
For $\ome=1$,  our functional equation reduces to the result of Sweet \cite{sweet}.
In section \ref{sec:3}, we explain the relation of the functional equation of the prehomogeneous vector space $\mathrm{Sym}_n(F)$ and that of the degenerate Whittaker functional of the degenerate principal series of $\mathrm{Sp}_n(F)$.
Note that this relation was established for unitary groups in Kudla and Sweet \cite{kudlasweet}.
Combining these results, we prove the functional equation of the Siegel series in section \ref{sec:4}.

I thank late Prof.~Hiroshi Saito for his kind advice.
I thank Prof.~Fumihiro Sato for his comment.

This research was partially supported by the JSPS KAKENHI Grant Number 26610005, 24540005.

\section{Weil constants and Tate's local factors}
\label{sec:1}

Let $F$ be a non-archimedean local field whose characteristic is not $2$.   
The maximal order of $F$ and its maximal ideal is denoted by $\frko$ and $\frkp$, respectively.
The number of elements of the residue field $\frkk=\frko/\frkp$ is denoted by $q$.
For $x\in F^\times$, we have $q^{-\mathrm{ord}x}=|x|$.
%Let $e$ be the integer such that $|2|^{-1}=q^e$.
The Haar measure $dx$ on $F$ is normalized so that $\int_{\frko}dx=1$.
The Hilbert symbol of $F$ of degree 2 is denoted by $\langle\,\,,\,\,\rangle$. 
We put $F^{\times 2}=\{x^2\,|\, x\in F^\times\}$.
Similarly, put put $\frko^{\times 2}=\{x^2\,|\, x\in \frko^\times\}$.
It is well-known that $[F^\times:F^{\times 2}]=4\,|2|^{-1}$ and $[\frko^\times:\frko^{\times 2}]=2\,|2|^{-1}$.
For $\theta\in F^\times/F^{\times 2}$, we put $\chi_\theta(x)=\langle \theta,x\rangle$.  

We fix a non-trivial additive character $\psi$ of $F$.  
Let $c_\psi$ be the order of $\psi$, i.~e., $c_\psi$ is the maximal integer $c$ such that $\psi$ is trivial on $\frkp^{-c}$.
We fix an element $\bdel\in F^\times$ such that $\mathrm{ord}(\bdel)=c_\psi$.

For each Schwartz function $\phi\in\scrs(F)$, the Fourier transform $\hat \phi$ is defined by
\[
\hat\phi(x)=|\bdel|^{1/2}\int_{F}\phi(y)\psi(xy)\, dy.
\]
Note that  the Haar measure $|\bdel|^{1/2}dy$ is the self-dual Haar measure for the Fourier transform $\phi\mapsto \hat\phi$.
For each $a\in F^\times$, there exists a constant $\alp_\psi(a)$, called the Weil constant, which satisfies
\begin{equation}
\label{wc}
\int_{F} \phi(x)\psi(ax^2)\,dx=\alp_\psi(a)|2a|^{-1/2}\int_{F}\hat\phi(x)\psi\left(-\frac{x^2}{4a}\right)\,dx
\end{equation}
for any $\phi\in\scrs(F)$ (cf. Weil \cite{Weil:64}).
The Weil constant $\alp_\psi(a)$ depends only on the class of $a$ in $F^\times/F^{\times 2}$, and so the symbol $\alp_\psi(\theta)$ for $\theta\in  F^\times/F^{\times 2}$ is meaningful.  
Clearly, we have $\alp_\psi(-a)=\overline{\alp_\psi(a)}$.
It is easy to see $\alp_{\psi_\xi}(a)=\alp_\psi(\xi a)$ for $\xi\in F^\times$, where $\psi_\xi(x)=\psi(\xi x)$.
For any $a, b\in F^\times$, we have
\[
\frac{\alp_\psi(a)\alp_\psi(b)}{\alp_\psi(1)\alp_\psi(ab)}=\langle a, b\rangle.
\]
If there is no fear of confusion, we write $\alp(a)$ for $\alp_\psi(a)$.

\begin{lemma}
\label{lem:1.1}
For $y\in F^\times$, 
\[
\sum_{x\in F^\times/F^{\times 2}}\alp(x)\langle x,y \rangle = 2|2|^{-1/2}
\frac{\alp(1)}{\alp(y)}.
\]
\end{lemma}
\begin{proof}
Since
\[
\sum_{x\in F^\times/F^{\times 2}}\alp(x)\langle x,y \rangle
=
\sum_{x\in F^\times/F^{\times 2}}\alp(x)
\frac{\alp(1)\alp(xy)}{\alp(x)\alp(y)}=\frac{\alp(1)}{\alp(y)}
\sum_{x\in F^\times/F^{\times 2}}\alp(x),
\]
it is enough to prove that
\[
\sum_{x\in F^\times/F^{\times 2}}\alp(x)=2|2|^{-1/2}.
\]
This was proved by Kahn \cite{kahn}.
We follow the argument of his first proof.
We may assume $\psi$ is of order 0.  
Let $\vpi$ be a prime element of $F$.
We choose a set $A$ of complete representatives of $\frko^\times/\frko^{\times 2}$.  Then $A\cup \vpi A$ is a set of complete representatives of $F^\times/F^{\times 2}$.
Then we have
\begin{align*}
\sum_{x\in A}\alp(x)=&\mathrm{Vol}(\frko^{\times 2})^{-1}\int_{x\in\frko^\times}\alp(x)\, dx,
\\
\sum_{x\in \vpi A}\alp(x)=&\mathrm{Vol}(\frko^{\times 2})^{-1}\int_{x\in\frko^\times}\alp(\vpi x)\, dx.
\end{align*}
Note that $\mathrm{Vol}(\frko^{\times 2})=[\frko^\times :\frko^{\times 2}]^{-1}\mathrm{Vol}(\frko^\times)=2^{-1}|2|(1-q^{-1})$.

Let $\phi_0$ be the characteristic function of $\frko$.
By putting $\phi=\phi_0$ in (\ref{wc}), we obtain 
\[
\alp(a)=|2a|^{-1/2} \int_{y\in \frko} \psi\left(\frac{y^2}{4a}\right)\, dy
\]
for $a\in\frko\setminus \{0\}$.
Then we have
\begin{align*}
\int_{x\in\frko^\times}\alp(x)\, dx
=&|2|^{-1/2}
\int_{\frko^\times} \int_{\frko} \psi(\frac{x y^2}4)\, dy\, dx \\
=&
|2|^{-1/2}
\left(
\int_{\frko} \int_{\frko} \psi(\frac{x y^2}4)\, dy\, dx-
\int_{\frkp} \int_{\frko} \psi(\frac{x y^2}4)\, dy\, dx
\right).
\end{align*}
Note that
\begin{align*}
\int_{\frko} \int_{\frko} \psi(\frac{x y^2}4)\, dy\, dx=&\mathrm{Vol}(\{y\in\frko\,|\, y^2\in4\frko\})=\mathrm{Vol}(2\frko)=|2|, \\
\int_{\frkp} \int_{\frko} \psi(\frac{x y^2}4)\, dy\, dx=&q^{-1}\mathrm{Vol}(\{y\in\frko\,|\, y^2\in4\frkp^{-1}\})=q^{-1}\mathrm{Vol}(2\frko)=|2|q^{-1}.
\end{align*}
It follows that
\[
\sum_{x\in A}\alp(x)=2|2|^{-1}(1-q^{-1})^{-1} |2|^{-1/2} |2|(1-q^{-1})=
2|2|^{-1/2}.
\]
Similarly, we have
\begin{align*}
\int_{x\in\frko^\times}\alp(\vpi x)\, dx
=&|2|^{-1/2}
\int_{\frko^\times} \int_{\frko} \psi(\frac{x y^2}{4\vpi})\, dy\, dx \\
=&
|2|^{-1/2}
\left(
\int_{\frko} \int_{\frko} \psi(\frac{x y^2}{4\vpi})\, dy\, dx-
\int_{\frkp} \int_{\frko} \psi(\frac{x y^2}{4\vpi})\, dy\, dx
\right).
\end{align*}
In this case, 
\begin{align*}
\int_{\frko} \int_{\frko} \psi(\frac{x y^2}{4\vpi})\, dy\, dx=&\mathrm{Vol}(\{y\in\frko\,|\, y^2\in4\frkp\})=\mathrm{Vol}(2\frkp)=|2|q^{-1}, \\
\int_{\frkp} \int_{\frko} \psi(\frac{x y^2}{4\vpi})\, dy\, dx=&q^{-1}\mathrm{Vol}(\{y\in\frko\,|\, y^2\in4\frko\})=q^{-1}\mathrm{Vol}(2\frko)=|2|q^{-1}.
\end{align*}
It follows that
\[
\sum_{x\in \vpi A}\alp(x)=0.
\]
Hence the lemma.
\end{proof}

For a character $\omega$ of $F^\times$, the $\vep$ and $L$ factor of $\omega$ are denoted by $\vep(s,\omega,\psi)$ and $L(s,\omega)$, respectively.  
We also use the notation 
\[
\vep'(s,\omega,\psi)=\vep(s,\omega,\psi)
\frac{L(1-s,\omega^{-1})}{L(s,\omega)}.
\]
Then Tate's local functional equation says
\[
\int_F \phi(x)\ome(x)|x|^{s-1}\,dx
=
\vep'(s,\ome, \psi)^{-1}
\int_F \hat\phi(x)\ome^{-1}(x)|x|^{-s}\,dx.
\]
It is well-known that 
\[
\vep'(s,\ome, \psi)\vep'(1-s,\ome^{-1}, \psi)=\ome(-1).
\]
When there is no fear of confusion, we write $\vep(s, \ome)$ (resp. $\vep'(s, \ome)$)  for $\vep(s, \ome, \psi)$ (resp. $\vep'(s, \ome, \psi)$).

\begin{lemma} 
\label{lem:1.2}
Let $\ome$ be a quasi-character of $F^\times$.
For $\rho\in F^\times$ and $\phi\in\scrs(F)$, we have
\begin{align*}
 &\int_{x\in \rho\cdot F^{\times 2}}\phi(x)\ome(x)|x|^{s-1}\,dx \\
=&4^{-1}|2|\sum_{\theta\in F^\times/F^{\times 2}}\chi_\theta(\rho)\vep'(s,\ome\chi_\theta)^{-1}\int_F \hat\phi(x)\ome^{-1}\chi_\theta(x)|x|^{-s}\,dx.
\end{align*}
\end{lemma}
\begin{proof}
Since $[F^\times:F^{\times 2}]=4|2|^{-1}$, we have
\begin{align*}
 &\int_{x\in \rho\cdot F^{\times 2}}\phi(x)\ome(x)|x|^{s-1}\,dx \\
=&4^{-1}|2|\sum_{\theta\in F^\times/F^{\times 2}}\chi_\theta(\rho)
\int_F \phi(x)\ome\chi_\theta(x)|x|^{s-1}\,dx.
\end{align*}
Then the lemma follows from Tate's functional equation.
\end{proof}

\begin{lemma} 
\label{lem:1.3}
Let $\ome$ be a quasi-character of $F^\times$.  Then we have
\[
\sum_{\tht\in F^\times/F^{\times 2}}\,{\frac{\alp(1)}{\alp(\tht)}}
\vep'(s,\ome\chi_\tht)^{-1}
=2\,|2|^{-2s}\,\ome^{-1}(4)
\vep'(2s,\ome^{2})^{-1}\vep'(s+\frac12,\ome).
\]
\end{lemma}
\begin{proof}
We follow the argument of  Rallis and Schiffmann [RS].
Choose $\vph\in\scrs(F)$ such that
\[
\vph(0)=\hat\vph(0)=0,
\]
and
\[
\int_{F^\times} \vph(x)\ome^2(x)|x|^{2s}\, d^\times\! x\not\equiv 0.
\]
Here, $d^\times\!x=|x|^{-1}dx$.
Put
\begin{align*}
\Phi_1(x)&=\begin{cases} 2|x|^{1/2}[\vph(\sqrt{x})+\vph(-\sqrt{x})] & \text{ if $x\in F^{\times 2}$,} \\
0& \text{ otherwise, }
\end{cases} \\
\Phi_2(x)&=\begin{cases} 2|x|^{1/2}[\hat\vph(\sqrt{x})+\hat\vph(-\sqrt{x})] & \text{ if $x\in F^{\times 2}$,} \\
0& \text{ otherwise. }
\end{cases}
\end{align*}
Then we have $\Phi_1, \Phi_2\in\scrs(F)$ and
\begin{align*} \int_F \Phi_1(y)\psi(xy)\, dy =& \int_F \vph(y)\psi(xy^2)\,dy \\
=&\alp(x)|2x|^{-{1/2}} \int_F \hat\vph(y)\psi(-\frac{y^2}{4x})\,dy \\
=&\alp(x)|2x|^{-{1/2}} \int_F \Phi_2(y)\psi(-\frac{y}{4x})\,dy.
\end{align*}
It follows that
\begin{equation}
\label{eq:2}
\widehat{\Phi_1}(x)=\alp(x)|2x|^{-{1/2}} \widehat{\Phi_2}(-\frac{1}{4x})
\end{equation}
We have a functional equation
\begin{align*}\int_{F^\times} \widehat{\Phi_1}(x) \ome^{-1}(x) |x|^{-s+(1/2)} \, d^\times \!x=&\vep'(s+\frac{1}{2},\ome)\int_{F^\times} \Phi_1(x) \ome(x) |x|^{s+(1/2)} \, {d}^\times \!x \\
=&\vep'(s+\frac{1}{2},\ome)\int_{F^\times} \vph(x) \ome^2(x) |x|^{2s} \, {d}^\times \!x \\ 
=&\vep'(s+\frac{1}{2},\ome)\vep'(2s, \ome^2)^{-1} \\
&\quad \times
\int_{F^\times} \hat\vph(x) \ome^{-2}(x) |x|^{1-2s} \, {d}^\times \!x. 
\end{align*}
By the equation (\ref{eq:2}), the left hand side is 
\begin{align*} &\int_{F^\times} \widehat{\Phi_1}(x) \ome^{-1}(x) |x|^{-s+(1/2)} \, {d}^\times \!x \\
=&|2|^{-{1/2}}\int_{F^\times} \alp(x)\widehat{\Phi_2}(-\frac{1}{4x}) \ome^{-1}(x) |x|^{-s} \, {d}^\times \!x \\
=&|2|^{2s-(1/2)}\ome(-4)\int_{F^\times} \overline{\alp(x)}\,\widehat{\Phi_2}(x) \ome(x) |x|^{s} \, {d}^\times \!x \\
=&|2|^{2s-(1/2)}\ome(-4)
\sum_{\bet\in\ktwo}\overline{\alp(\bet)} \int_{\bet\cdot F^{\times 2}} \widehat{\Phi_2}(x) \ome(x) |x|^{s} \, {d}^\times \!x 
\end{align*}
Since $[F^\times:F^{\times 2}]=4|2|^{-1}$, we have
\begin{align*} 
& \sum_{\bet\in\ktwo}\overline{\alp(\bet)}\int_{\bet\cdot F^{\times 2}} \widehat{\Phi_2}(x) \ome(x) |x|^{s} \, {d}^\times \!x \\
=&4^{-1}|2|\sum_{\bet,\tht\in\ktwo} \overline{\alp(\bet)}\chi_\bet(\tht) \int_{F^\times} \widehat{\Phi_2}(x) \ome\chi_\tht(x) |x|^{s} \, {d}^\times \!x.
\end{align*}
By Lemma \ref{lem:1.1} and Tate's functional equation, we have
\begin{align*}
&4^{-1}|2|\sum_{\bet,\tht\in\ktwo} \overline{\alp(\bet)}\chi_\bet(\tht) \int_{F^\times} \widehat{\Phi_2}(x) \ome\chi_\tht(x) |x|^{s} \, {d}^\times \!x \\
=&2^{-1}|2|^{1/2}\sum_{\tht\in\ktwo} {\frac{\alp(\tht)}{\alp(1)}} \int_{F^\times} \widehat{\Phi_2}(x) \ome\chi_\tht(x) |x|^{s} \, {d}^\times \!x \\
=&2^{-1}|2|^{1/2}\sum_{\tht\in\ktwo} {\frac{\alp(\tht)}{\alp(1)}} 
\vep'(1-s,\ome^{-1}\chi_\tht,\psi)
\int_{F^\times} \Phi_2(x) \ome^{-1}\chi_\tht(x) |x|^{1-s} \, {d}^\times \!x \\
=&2^{-1}|2|^{1/2}\ome(-1)\sum_{\tht\in\ktwo} {\frac{\alp(\tht)}{\alp(1)}}\chi_\tht(-1) 
\vep'(s,\ome\chi_\tht,\psi)^{-1} \\
&\quad\times\int_{F^\times} \hat\vph(x) \ome^{-2}(x) |x|^{1-2s} \, {d}^\times \!x \\
=&2^{-1}|2|^{1/2}\ome(-1)\sum_{\tht\in\ktwo} {\frac{\alp(1)}{\alp(\tht)}}\vep'(s,\ome\chi_\tht,\psi)^{-1} \\
&\quad\times\int_{F^\times} \hat\vph(x) \ome^{-2}(x) |x|^{1-2s} \, {d}^\times \!x.
\end{align*}
This proves the lemma.
\end{proof}

\begin{lemma} 
For $\phi\in \scrs(F)$, 
\begin{align*}
 & \int_F \phi(x) \overline{\alp(x)}\ome(x)|x|^{s-1}\,d x \\
= & |2|^{-2s+(1/2)}\overline{\alp(1)}\ome^{-1}(4)\vep'(2s,\ome^2, \psi)^{-1} \\
&\quad\times
\sum_{\beta \in F^{\times}/ F^{\times 2}} \overline{\alp(\beta)} \vep'(s+\frac{1}{2}, \ome\chi_{-\beta}, \psi) 
\int_{x\in\beta \cdot F^{\times 2}}  \hat\phi(x) \ome^{-1}(x)|x|^{-s}\,d x 
\end{align*}
\end{lemma}
\medskip

\begin{proof}
By Lemma \ref{lem:1.2}, we have
\begin{align*} 
\int_F \phi(x) \overline{\alp(x)}\ome(x)|x|^{s-1}\,d x 
=&\sum_{\rho\in \ktwo} \overline{\alp(\rho)} \int_{x\in \rho\cdot F^{\times 2}}  \phi(x) \ome(x) |x|^{s-1} \,d x \\
=& 4^{-1}|2| \sum_{\rho,\tht,\beta\in \ktwo} \overline{\alp(\rho)}  \, \chi_\tht(\rho\beta) \, \vep'(s, \ome\chi_\tht)^{-1} \\
&\; \times
\int_{x\in \beta\cdot  F^{\times 2}} \hat\phi(x) \ome^{-1}(x) |x|^{-s}\,d x.
\end{align*}
By Lemma \ref{lem:1.1} and Lemma \ref{lem:1.3}, we have
\begin{align*}
&\sum_{\rho,\tht,\beta\in \ktwo} \overline{\alp(\rho)}  \, \chi_\tht(\rho\beta) \, \vep'(s, \ome\chi_\tht)^{-1} \\
=& 2|2|^{-1/2} 
\sum_{\tht,\beta \in \ktwo} 
\frac{\alp(\tht)}{\alp(1)}
 \, \chi_\tht(\beta)\,  \vep'(s, \ome\chi_\tht)^{-1} 
%\int_{x\in \beta\cdot  F^{\times 2}} \hat\phi(x) \ome^{-1}(x) |x|^{-s}\,d x 
\\
=& 2|2|^{-1/2} \overline{\alp(1)} \sum_{\tht,\beta\in \ktwo} 
\frac{\alp(1)}{\alp(-\beta\tht)\alp(\beta)}  
\vep'(s, \ome\chi_\tht)^{-1} 
\\
=& 2|2|^{-1/2} \overline{\alp(1)} \sum_{\tht,\beta\in \ktwo} 
\frac{\alp(1)}{\alp(\tht)\alp(\beta)}  
\vep'(s, \ome\chi_{-\beta\tht})^{-1} 
\\
=& 4 |2|^{-2s-(1/2)} \overline{\alp(1)}\ome^{-1}(4)\vep'(2s,\ome^2)^{-1}\\
&\quad\times
\sum_{\beta\in F^{\times}/F^{\times 2}} 
{\overline{\alp(\beta)}} \vep'(s+\frac{1}{2}, \ome\chi_{-\beta}).
\end{align*}
This proves the lemma.
\end{proof}

\section{The space of symmetric matrices of rank $n$.}
\label{sec:2}

Let $V=\mathrm{Sym}_n(F)$.
Then $V$ is a prehomogeneous vector space under the action of an algebraic group $\GL_n$.  
The set of open orbits of $V$ is denoted by $\calo$. 
The set $\vsst=\mathrm{Sym}_n(F)^{\mathrm{ss}}$ of semi-stable elements of $V$ is the union of all the open orbits.
Then $Q\in \vsst$ if and only if $\det Q\neq 0$.
The open orbit containing $Q\in\vsst$ is denoted by $V_Q$.
For each $Q\in \vsst$, we put 
$D_Q=(-4)^{[{n/2}]}\det Q$.  
\begin{definition}
Let $Q\in\vsst$.
The Clifford invariant $\eta_Q$ is the Hasse invariant of the Clifford algebra (resp.~the even Clifford algebra) of $Q$ if $n$ is even (resp.~odd).
\end{definition}

\begin{lemma}
Assume that $Q\in\vsst$ is equivalent to  $\mathrm{diag}(q_1,\ldots, q_n)$.
If $n=2m+1$ is odd, then
\[
\eta_Q=\langle -1,-1\rangle^{{m(m+1)}/2}\langle\, (-1)^m,\det Q \rangle
\eps_Q.
\]
If $n=2m$ is even, then
\[
\eta_Q=\langle -1,-1\rangle^{{m(m-1)}/2}\langle\, (-1)^{m+1},\det Q \rangle
\eps_Q.
\]
Here, $\eps_Q=\prod_{1\leq i< j\leq n}\langle q_i, q_j\rangle$.
\end{lemma}
\begin{proof}
See Scharlau \cite{scharlau} Ch.~\!9, Remark 2.~\!\!12, p333.
\end{proof}

If $Q$ is equivalent to $\mathrm{diag}(q_1,\cdots, q_n)$, we set 
\[
\alp_Q(a)=\alp_{aQ}(1)=\prod_{i=1}^n \alp(q_i a).
\] 
Then we have
\begin{align}
\label{wcq}
\int_{F^n} \phi(x)\psi(Q[x])\,dx=\frac{\alp_Q(1)}{|2^n\det Q|^{1/2}} \int_{F^n}\hat\phi(x)
\psi\left(-\frac{Q^{-1}[x]}{4}\right)\,dx
\end{align}
Here, as usual, $Q[x]={}^t\!xQx$ and
\[
\hat\phi(x)=|\bdel|^{n/2}
\int_{F^n}\phi(y)\psi({}^t\!y\cdot x)\, dy.
\]

\begin{lemma}
\label{lem:2.1}
\begin{itemize}
\item[(1)]
For $Q\in \vsst$ and $x\in F^\times$, we have
\[
\alp_Q(x)=
\begin{cases}
\displaystyle
\frac{\alp(D_Q)}{\alp(1)} \eta_Q\,\alp(x)\chi_{D_Q}(x) & \text{ if $n$ is odd,} \\
\noalign{\vskip 6pt}
\displaystyle\frac{\alp(1)}{\alp(D_Q)} \eta_Q\,\chi_{D_Q}(x) & \text{ if $n$ is even.} 
\end{cases}
\]
\item[(2)]
For $Q\in\vsst$ and $\beta\in F^\times$, put $Q\oplus \beta=\begin{pmatrix} Q & 0 \\ 0 & \beta \end{pmatrix}$.
Then we have
\[
\eta_{Q\oplus \beta}=
\begin{cases}
\eta_Q\chi_{D_Q}(\beta) & \text{ if $n$ is odd,} \\
\eta_Q\,\chi_{D_Q}(-\beta) & \text{ if $n$ is even.} 
\end{cases}
\]
\end{itemize}
\end{lemma}
\begin{proof}
We first prove (1).
Since $\frac{\alp(ax)}{\alp(a)}=\frac{\alp(x)}{\alp(1)}\langle a, x\rangle$, we have
\[
\frac{\alp_Q(x)}{\alp_Q(1)}=
\begin{cases}
\displaystyle
\frac{\alp(x)}{\alp(1)}\chi_{D_Q}(x) & \text{ if $n$ is odd,} \\ 
\noalign{\vskip 2pt}
\displaystyle
\chi_{D_Q}(x) & \text{ if $n$ is even} 
\end{cases}
\]
Thus we may assume $x=1$.
One can easily show 
\[
\alp_Q(1)=\alp(1)^{n-1}\alp(\det Q)\eps_Q,
\] 
where $\eps_Q=\prod_{1\leq i< j\leq n}\langle q_i, q_j\rangle$.
It is also easy to see that
\[
\alp(1)^{2m+1}\overline{\alp((-1)^m)}=\langle -1, -1 \rangle^{{m(m+1)}/2}.
\]
It follows that
\begin{align*}
\alp(-1)^{2m}\alp(a)
=&\langle -1, -1 \rangle^{{m(m+1)}/2}\overline{\alp(1)}\alp((-1)^m)\alp(a)
\\
=&\langle -1, -1 \rangle^{{m(m+1)}/2}\langle (-1)^m, a\rangle \alp((-1)^m a).
\end{align*}
If $n=2m+1$ is odd, we have
\begin{align*}
 \alp_Q(1)
=&\alp(1)^{2m}\alp(\det Q)\eps_Q \\                  
=&\langle -1,-1\rangle^{{m(m+1)}/2} \, \langle\, (-1)^m,\det Q \rangle \alp(D_Q) \eps_Q \\
=&
\eta_Q\alp(D_Q).
\end{align*}
Similarly, if $n=2m$ is even, we have
\begin{align*} 
\alp_Q(1) 
=&\alp(1)^{2m-1}\alp(\det Q)\eps_Q \\
=&\langle -1,-1\rangle^{m(m-1)/2} \, \langle\, (-1)^{m-1},\det Q \rangle \alp((-1)^{m-1}\det Q)\alp(1) 
\eps_Q \\
=&\eta_Q
\frac{\alp(1)}{\alp(D_Q)} . 
\end{align*} 
This proves (1).
If $n=2m+1$ is odd, we have
\begin{align*}
\eta_{Q\oplus\beta}=&
\langle -1,-1\rangle^{m(m+1)/2} \, \langle\, (-1)^{m},\beta\det Q \rangle\, \eps_Q \langle \det Q, \beta\rangle \\
=&\eta_Q\chi_{D_Q}(\beta).
\end{align*}
If $n=2m$ is even, we have
\begin{align*}
\eta_{Q\oplus\beta}=&
\langle -1,-1\rangle^{m(m+1)/2} \, \langle\, (-1)^m ,\beta\det Q \rangle\, \eps_Q \langle \det Q, \beta\rangle \\
=&\eta_Q\chi_{D_Q}(-\beta).
\end{align*}
Thus the lemma is proved.
\end{proof}

For $\Phi\in \scrs(V)$, we put
\[
\widehat\Phi(x)=\int_V \Phi(y)\psi(\mathrm{tr}(xy))dy.
\]
Here the measure $dx$ is the self dual measure for the Fourier transform, i.e., $dx=|2|^{{{n(n-1)}/4}}|\bdel|^{n(n+1)/2}\prod_{i=1}^n dx_{ii} \prod_{i<j}dx_{ij}$. 
Put $\sig=(n+1)/2$.
From the prehomogeneity of $V$, there is a meromorphic function $c_Q(\ome,s)$ such that
\begin{align*} 
&\int_{V} \Phi(X)\ome(\det X)\,|\det X|^{s-\sig}\,d X \\
=&\sum_{Q\in\calo}c_Q(\ome,s)\int_{V_Q}\widehat\Phi(X) \ome^{-1}(\det X) |\det X|^{-s}\,d X.
\end{align*}
Without a loss of generality, we may assume $\ome$ is unitary.
The left hand side is absolutely convergent for $\mathrm{Re}(s)>(n-1)/2$.

\begin{theorem}  
\label{thm:2.1}
If $n=2m+1$, then we have
\begin{align*}
 c_Q(\ome, s)=& \vep'(s-m,\ome)^{-1}\prod_{r=1}^{m}\vep'(2s-2m-1+2r,\ome^2)^{-1}  \\
&\times |2|^{-2ms+{{m(2m+1)}/2}} \ome^{-m}(4) \eta_Q.
\end{align*}
If $n=2m$, then we have
\begin{align*} 
c_Q(\ome, s)=&  \vep'(s-m+\thalf,\ome)^{-1}\prod_{r=1}^{m}\vep'(2s-2m+2r,\ome^2)^{-1} \\
 & \times |2|^{-2ms+{{m(2m-1)}/2}} \ome^{-m}(4) 
\frac{\alp(D_Q)}{\alp(1)} 
\vep'(s+\frac{1}{2}, \ome\chi_{D_Q}).
\end{align*}

\end{theorem}

\begin{proof}

\bigskip
We proceed by induction on $n$.  
When $n=1$, the theorem follows immediately from Tate's local functional equation.

Now we assume $n\geq1$. We assume the functional equation is true for $n$.  We put $V=\mathrm{Sym}_{n}(F)$, $V'=\mathrm{Sym}_{n+1}(F)$, $\sig={(n+1)/2}$, and $\sig'={(n+2)/2}$.  We denote the set of $\GL_n$ orbits of $V$ by $\calo$ and the set of $\GL_{n+1}$ orbits of $V'$ by $\calo'$.  

We may assume $\Phi\left(\begin{pmatrix} T&y \\ ^t\! y&x\end{pmatrix}\right)=\phi_1(T)\phi_2(y)\phi_3(x)$ for some $\phi_1\in\scrs(V)$, $\phi_2\in\scrs(F^{n})$, $\phi_3\in\scrs(F)$.
Note that 
\[
\widehat\Phi\left(\begin{pmatrix} T&y \\ ^t\! y&x \end{pmatrix}\right) =|2|^{{n}/2} \hat\phi_1(T)\hat\phi_2(2y)\hat\phi_3(x).
\]
We write $T[y]={}^t\;\!\!y T y$ for $T\in\mathrm{Sym}_n(F)$ and $y\in F^n$.
Then we have
\begin{align*} 
&\int_{T'\in V'} \Phi(T')\ome(\det T')|\det T'|^{s-{\sig'}} d T'  \\
=&\int_{T\in V}\int_{y\in F^n}\int_{x\in F}  \Phi\left(\begin{pmatrix} T+x^{-1}
y\;\!\!\cdot^t\! y &y \\  {}^t\! y&x \end{pmatrix}\right) \\
&\qquad\qquad\times
\ome(x\det T)|x\det T|^{s-{\sig'}}d x\,d y\,d T.
\end{align*}
Consider the integral 
\begin{align*}
&\int_{T\in V}\int_{y\in F^n}\int_{x\in F}  \Phi\left(\begin{pmatrix} T+x^{-1}y\;\!\!\cdot^t\! y &y \\  {}^t\! y&x \end{pmatrix}\right) \\
&\qquad\qquad\times
\ome(x \det T)|\det T|^{s_1-{\sig'}} |x|^{s_2-{\sig'}}\, d x\,d y\,d T .
\end{align*}
This integral is absolutely convergent for $\mathrm{Re}(s_1), \mathrm{Re}(s_2)>(n-1)/2$.
By the functional equation for $V$, this is equal to
\begin{align*}
&\sum_{Q\in \calo}c_Q(\ome,s_1-\frac{1}{2})
\int_{T\in {V_Q}}\int_{y\in F^n}\int_{x\in F} \hat\phi_1(T)\psi(-x^{-1}T[y]) \phi_2(y) \phi_3(x) \\
&\quad \times \ome^{-1}(\det T) \ome(x) |x|^{s_2-{\sig'}}\,|\det T|^{-s_1+(1/2)}\,d x \,d y \,d T .
\end{align*}
This integral is absolutely convergent for $\mathrm{Re}(s_1)<1/2$ 
and $\mathrm{Re}(s_2)>(n-1)/2$.
On this absolute convergence domain, one can change the order of the integration.
Using the equation (\ref{wcq}), this is equal to
\begin{align*}
& \sum_{Q\in \calo}c_Q(\ome,s_1-\frac{1}{2})
\int_{T\in V_Q}\int_{y\in F^n}\int_{x\in F} \\
&\quad\times \overline{\alp_Q(x)}|2^n x^{-n}\det T|^{-{1/2}}
\hat\phi_1(T)\hat\phi_2(y)\phi_3(x) \\
& \quad \times \psi(4^{-1}x T^{-1}[y])\ome^{-1}(\det T)|\det T|^{-s_1+(1/2)}\ome(x)|x|^{s_2-{\sig'}} \, d x \, d y \, d T \\
=&  \sum_{Q\in \calo} \sum_{\rho\in F^\times/F^{\times 2}}
c_Q(\ome,s_1-\frac{1}{2})\overline{\alp_Q(\rho)}
\int_{T\in V_Q}\int_{y\in F^n}\int_{x\in \rho \cdot F^{\times 2}}
|2|^{{n/2}}\hat\phi_1(T)\hat\phi_2(2y)\phi_3(x) \\
& \quad \times \psi(xT^{-1}[y]) \ome(\det T) |\det T|^{-s_1} \ome(x) |x|^{s_2-1} \, d x \, d y \, d T .
\end{align*}
By Lemma \ref{lem:1.2}, this is equal to
\begin{align*}
&  \sum_{Q\in \calo} \sum_{\rho,\tht,\beta\in F^\times/F^{\times 2}}
4^{-1} |2| c_Q(\ome,s_1-\frac{1}{2})\overline{\alp_Q(\rho)}\, \chi_ \tht(\rho\beta)\, \vep'(s, \ome\chi_\tht)^{-1}  \\
& \quad \times \int_{T\in V_Q}\int_{y\in F^n}\int_{x\in \beta\cdot F^{\times 2}} |2|^{{n/2}}\hat\phi_1(T)\hat\phi_2(2y)\hat\phi_3(x+T^{-1}[y]) \\
& \quad \times \ome^{-1}(x\det T) |\det T|^{-s_1} |x|^{-s_2} \, d x \, d y \, d T .
\end{align*}
By putting $s_1=s_2=s$, we have
\[
c_{Q\oplus\beta}(\ome, s) = 4^{-1} |2| c_Q(\ome, s-\frac{1}{2}) 
\sum_{\rho, \tht\in F^\times/F^{\times 2}} \overline{\alp_Q(\rho)} \, \chi_\tht(\rho\beta)\, \vep'(s, \ome\chi_\tht)^{-1}.
\]

If $n=2m$ is even, by using Lemma \ref{lem:2.1}, we have
\begin{align*} &4^{-1} |2|  \sum_{\rho, \tht\in F^\times/F^{\times 2}} \overline{\alp_Q(\rho)}  \, \chi_\tht( \rho\beta)\, \vep'(s, \ome\chi_\tht)^{-1} \\
=& 4^{-1} |2| \,\eta_Q\,
\frac{\alp(D_Q)}{\alp(1)}
 \sum_{\rho, \tht\in F^\times/F^{\times 2}} 
\chi_{D_Q}( \rho ) \,
 \chi_\tht(\rho\beta)\, \vep'(s, \ome\chi_\tht)^{-1} \\
=& \eta_Q\, \frac{\alp(1)}{\alp(D_Q)}
\chi_{D_Q}(  -\beta ) \, \vep'(s, \ome\chi_{D_Q})^{-1} \\
=& \eta_{Q\oplus\beta}
\frac{\alp(1)}{\alp(D_Q)} 
 \, \vep'(s, \ome\chi_{D_Q})^{-1}.
\end{align*}
It follows that
\begin{align*}
 c_{Q\oplus\beta}(\ome, s) =& c_Q(\ome, s-\frac{1}{2})  \eta_{Q\oplus\beta} \, \frac{\alp(1)}{\alp(D_Q)} 
\vep'(s, \ome\chi_{D_Q})^{-1} \\
=&  \vep'(s-m,\ome)^{-1}\prod_{r=1}^{m}\vep'(2s-2m-1+2r,\ome^2)^{-1} \\
 & \times |2|^{-2ms+(m(2m+1))/2)} \ome^{-m}(4) \,\eta_{Q\oplus\beta} .
\end{align*}
On the other hand, if $n=2m+1$ is odd, by using Lemma \ref{lem:2.1} (1), we have
\begin{align*} &4^{-1} |2|  \sum_{\rho, \tht\in F^\times/F^{\times 2}} \overline{\alp_Q(\rho)}  \, \chi_\tht( \rho\beta)\, \vep'(s, \ome\chi_\tht)^{-1} \\
=&4^{-1} |2|   \frac{\alp(1)}{\alp(D_Q)} 
\eta_Q\,
 \sum_{\rho, \tht\in F^\times/F^{\times 2}} 
\overline{\alp(\rho)} \, \chi_{D_Q}(  \rho ) 
 \chi_\tht(\rho\beta)\, \vep'(s, \ome\chi_\tht)^{-1} \\
=& 4^{-1} |2|  \frac{\alp(1)}{\alp(D_Q)} 
\eta_Q\,
 \sum_{\rho, \tht\in F^\times/F^{\times 2}} 
\overline{\alp(\rho)} \, \chi_{\rho}(D_Q\tht) \, 
 \chi_\tht( \beta ) \, \vep'(s, \ome\chi_\tht)^{-1} .
\end{align*}
By Lemma \ref{lem:1.1} and Lemma \ref{lem:1.3}, we have
\begin{align*}
& 4^{-1} |2|  \frac{\alp(1)}{\alp(D_Q)} 
\eta_Q\,
 \sum_{\rho, \tht\in F^\times/F^{\times 2}} 
\overline{\alp(\rho)} \, \chi_{\rho}(D_Q\tht) \, 
 \chi_\tht( \beta ) \, \vep'(s, \ome\chi_\tht)^{-1} \\
=& 2^{-1}|2|^{1/2}  \overline{\alp(D_Q)} \eta_Q\,
\sum_{\tht\in F^\times/F^{\times 2}} 
\alp(D_Q\tht) \,  
 \chi_\tht( \beta ) \, \vep'(s, \ome\chi_\tht)^{-1} \\
=&2^{-1}|2|^{1/2}  \overline{\alp(D_Q)} \eta_Q\,
 \sum_{\tht\in F^\times/F^{\times 2}} 
\alp(-\beta\tht) \,  
 \chi_{D_Q\tht}( \beta ) \, \vep'(s, \ome\chi_{-D_Q\beta\tht})^{-1} \\
=&2^{-1}|2|^{1/2}   \overline{\alp(D_Q)\alp(\beta)} \, \chi_{ D_Q}( \beta ) \,\eta_Q\,
 \sum_{\tht\in F^\times/F^{\times 2}} 
\frac{\alp(1)}{\alp(\tht)} 
\, \vep'(s, \ome\chi_{-D_Q\beta\tht})^{-1} \\
=& |2|^{-2s+(1/2)} \ome^{-1}(4) 
\frac{\alp(D_{Q\oplus\beta})}{\alp(1)} 
\eta_Q\,
 \vep'(2s,\ome^2)^{-1}\vep'(s+\frac{1}{2}, \ome\chi_{D_{Q\oplus\beta}}) .
\end{align*}
It follows that
\begin{align*} c_{Q\oplus\beta}(\ome,s)=&c_Q(\ome,s-\frac{1}{2}) |2|^{-2s+(1/2)} \ome^{-1}(4) 
\frac{\alp(D_{Q\oplus\beta})}{\alp(1)} \eta_Q
 \\
&\times  
\vep(2s,\ome^2)^{-1} \vep'(s+\frac{1}{2}, \ome\chi_{D_{Q\oplus\beta}}) \\
=&  \vep'(s-m-\frac{1}{2},\ome)^{-1}\prod_{r=1}^{m+1}\vep'(2s-2m-2+2r,\ome^2)^{-1} \\
 & \times |2|^{-2(m+1)s+({(m+1)(2m+1)}/2)}\ome^{-m-1}(4) \frac{\alp(D_{Q\oplus\beta})}{\alp(1)}  
\vep'(s+\frac{1}{2}, \ome\chi_{D_{Q\oplus\beta}}). \\
\end{align*}
Thus we have proved Theorem \ref{thm:2.1}
\end{proof}

By the theory of prehomogeneous vector space, there is a function $c'_Q(\ome, s)$ such that
\begin{align*}
& \int_{V} \Phi(X)  \ome(\det X)\, \eta_X \,|\det X|^{s-\sig}\,d X \\
=&\sum_{Q\in\calo}c'_Q(\ome,s)\int_{V_\eta}\widehat\Phi(X) \ome^{-1}(\det X) \, |\det X|^{-s}\,d X.
\end{align*}

\begin{theorem}
\label{thm:2.2}

If $n=2m+1$ is odd, then we have
\begin{align*} 
c'_Q(\ome, s)=& 
\vep'(s,\ome)^{-1}\prod_{r=1}^{m}\vep'(2s-2m-2+2r,\ome^2)^{-1}\\
&\times |2|^{-2ms+(m(2m+3)/2)} \ome^{-m}(4).
\end{align*}

If $n=2m$ is even, then we have
\begin{align*}
 c'_Q(\ome, s)=& 
\prod_{r=1}^{m}\vep'(2s-2m-1+2r,\ome^2)^{-1}\\
&\times |2|^{-2ms+(m(2m+1)/2)} \ome^{-m}(4)
\frac{\alp(1)}{\alp(D_Q)} 
 \eta_Q 
\end{align*}
\end{theorem}
This theorem can be proved in the same way as Theorem \ref{thm:2.1}.
As we just give an outline of the proof.
\begin{proof}
Assume $n=2m+1$ is odd.
By Theorem \ref{thm:2.1}, we have
\begin{align*}
&\int_{V} \Phi(X)\ome(\det X)\,|\det X|^{s-m-1}\,d X \\
=&
|2|^{-2ms+(m(2m+1)/2)} \ome^{-m}(4) 
\vep'(s-m,\ome)^{-1}\prod_{r=1}^{m}\vep'(2s-2m-1+2r,\ome^2)^{-1}  \\
&\quad\times
\sum_{Q\in\calo}
\int_V
\widehat\Phi(X) \ome^{-1}(\det X) \eta_X |\det X|^{-s}\,d X.
\end{align*}
By changing $\ome$, $s$ and $\Phi$ by $\ome^{-1}$, $m+1-s$ and $\widehat\Phi$, respectively, we obtain the desired functional equation for odd $n$.
Assume $n$ is odd.
Then as in the proof of Theorem \ref{thm:2.1}, one can prove
\begin{align*}
c'_{Q\oplus \beta}(\ome, s)=&
4^{-1}|2|
\sum_{\rho, \tht\in F^\times/F^{\times 2}} 
c'_Q(\ome\chi_{\rho}, s-\frac12)
\overline{\alp_Q(\rho)} \\
&\quad\times
\chi_{\rho}(\det Q)\chi_\tht(\rho\beta)
\vep'(s, \ome\chi_{(-1)^{m}\tht})^{-1}.
\end{align*}
After a little calculation, we obtain the desired functional equation for even $n$.
\end{proof}

\begin{remark}
The method of the proofs above are due to F.~Sato.
See also Muller \cite{muller}.
Sweets \cite{sweet} calculated the ``gamma matrix'' for the prehomogeneous vector space $V$ in the case $\ome=\mathbf{1}$.
Theorem \ref{thm:2.1} for $\ome=\mathbf{1}$ follows from his results.
\end{remark}

\section{\bf Degenerate Whittaker functionals}
\label{sec:3}

In this section, we follow the argument of Kudla and Sweet \cite{sweet}.
Let 
\[
\SP_n(F)=\left\{\begin{pmatrix} A & B \\ C & D\end{pmatrix}\in\mathrm{M}_n(F)\;\vrule \; A \,{}^t\! D - B \,{}^t\;\!\! C=\mathbf{1}_n\right\}
\]
be the symplectic group of rank $n$.
Let 
\[
P_n(F)=\left\{\begin{pmatrix} A & B \\ 0 & {}^t\!A^{-1}\end{pmatrix}\in\mathrm{Sp}_n(F)\right\}
\]
be the Siegel parabolic subgroup of $\SP_n(F)$.
Put 
\[
\bfn(B)=\begin{pmatrix} \mathbf{1}_n & B \\ 0 & \mathbf{1}_n \end{pmatrix}\in \SP_n(F)
\]
for $B\in\mathrm{Sym}_n(F)$.
Then $N_n(F)=\{\bfn(B)\,|\, B\in \mathrm{Sym}_n(F)\}$ is the unipotent radical of $P_n(F)$.

For a quasi-character $\ome$ of $F^\times$, we consider the degenerate principal series $I(\ome, s)=\mathrm{Ind}_{P_n}^{\SP_n}(\ome|\det|^s)$.
The space of $I(\ome, s)$ consists of locally constant function $f(g)$ on $\SP_n(F)$ such that
\[
f\left(
\begin{pmatrix} A & B \\ 0 & {}^t\! A^{-1}\end{pmatrix}g
\right)
=\ome(\det A)|\det A|^{s+((n+1)/2)} f(g)
\]
for any $\begin{pmatrix} A & B \\ 0 & {}^t\! A^{-1}\end{pmatrix}\in P_n(F)$ and $g\in \SP_n(F)$.
For $f(g)\in I(\ome, s)$ and  $B\in \mathrm{Sym}_n(F)^{\mathrm{ss}}$, put
\[
M(s)f(g )=\int_{\mathrm{Sym}_n(F)} f( w_n \bfn(x) g)\, dx.
\]
\[
\mathrm{Wh}_B(s)f=\int_{\mathrm{Sym}_n(F)} f( w_n \bfn(x) ) \overline{\psi(\mathrm{tr}(Bx))} \, dx.
\]
Here,
\[
w_n=\begin{pmatrix} 0  & -{\bf1}_n \\ {\bf1}_n & 0 \end{pmatrix}.
\]
The integrals $M(s)$ and $\mathrm{Wh}_B(s)$ are absolutely convergent for $\mathrm{Re}(s)\gg 0$ and can be meromorphically continued to the whole complex plane.
If $s$ is not a pole of $M(s)$, then $M(s)f(g)\in I(\ome^{-1}, -s)$.
Moreover, it is known that $\mathrm{Wh}_B(s)$ is entire.

\begin{lemma}
\label{lem:4.3}
The following functional equation holds:
\[
\mathrm{Wh}_B(-s)\circ M(s)=\ome^{-1}(\det B)|\det B|^{-s} c_B(\ome, s) \mathrm{Wh}_B(s).
\]
\end{lemma}

\begin{proof}
Let $m$ be a sufficiently large integer such that
\[
B+\frkp^m\mathrm{Sym}_n(\frko)\subset \{x\in\mathrm{Sym}_n(F)_B\, |\, |\det x|=|\det B|, \;\ome(\det x)=\ome(\det B) \}.
\] 
Here, $\mathrm{Sym}_n(F)_B=\{B[X]\,|\, X\in \GL_n(F)\}$ is the orbit containing $B$.
Note that $\mathrm{Sym}_n(F)_B$ is an open subset of $\mathrm{Sym}_n(F)$.
Let $\Phi\in\scrs(\mathrm{Sym}_n(F))$ be the characteristic function of $B+\frkp^m\mathrm{Sym}_n(\frko)$.
We define $f_\Phi\in I(\ome, s)$ such that
\begin{itemize}
\item $\mathrm{Supp}(f_\Phi)\subset P_n(F)wN_n(F)$.
\item $f(w_n\bfn(x))=\widehat\Phi(x)$ for $x\in \mathrm{Sym}_n(F)$.
\end{itemize}
Then, we have $\mathrm{Wh}(s)f_\Phi=\widehat{\widehat\Phi}(-B)\neq 0$.
On the other hand,  we have
\begin{align*}
M(s)f_\Phi(w_n\bfn(x))=
&\int_{y\in\mathrm{Sym}_n(F)} f_\Phi(w_n\bfn(y)w_n\bfn(x))\, dx  \\
=&\int_{y\in\mathrm{Sym}_n(F)} \widehat\Phi(x-y^{-1})\ome^{-1}(\det y)|\det y|^{-s-((n+1)/2)}\, dy \\
=&\int_{y\in\mathrm{Sym}_n(F)} \widehat\Phi(x-y)\ome(\det y)|\det y|^{s-((n+1)/2)}\, dy.
\end{align*}
By Theorem \ref{thm:2.1}, this is equal to
\begin{align*}
&\sum_{B\in\calo} c_B(\ome, s)
\int_{y\in \mathrm{Sym}_n(F)_B}  \Phi(y)\psi(\mathrm{tr}(xy)) \ome^{-1}(\det y) |\det y|^{-s} \, dy \\
=&
 c_B(\ome, s)\ome^{-1}(\det B) |\det B|^{-s} 
\int_{y\in \mathrm{Sym}_n(F)}  \Phi(y)\psi(\mathrm{tr}(xy)) \, dy \\
=&
 c_B(\ome, s)\ome^{-1}(\det B) |\det B|^{-s} 
\widehat\Phi(x).
\end{align*}
It follows that
\begin{align*}
\mathrm{Wh}(-s)M(s)f_\Phi=&\int_{x\in\mathrm{Sym}_n(F)} M(s)f_\Phi(w_n\bfn(x)) \overline{\psi_B(x)}\, dx \\
=&c_B(\ome, s)\ome^{-1}(\det B) |\det B|^{-s}\widehat{\widehat\Phi}(-B) .
\end{align*}
Hence the lemma.
\end{proof}

\section{\bf Siegel series and its functional equation}
\label{sec:4}
As before, let $F$ be a non-archimedean local field.
In this section, we assume that the additive character $\psi$ is of order $0$.
Recall that $B\in \mathrm{Sym}_n(F)$ is called non-degenerate if $D_B\neq 0$.
The set of non-degenerate elements in $\mathrm{Sym}_n(F)$ is denoted by $\mathrm{Sym}_n(F)^{\mathrm{ss}}$.
For $B\in\mathrm{Sym}_n(F)^{\mathrm{ss}}$, put
\begin{align*}
D_B=&
(-4)^{[n/2]}\det B, \\
\xi_B=&
\begin{cases}
\langle D_B, \vpi\rangle & \text{ if $F(\sqrt{D_B})/F$ is unramified,} \\
0 & \text{ otherwise.}
\end{cases}
\end{align*}
Let $\frkd_B$ be the conductor of the  extension $F(\sqrt{D_B})/F$.
We set
\[
\del_B=(\mathrm{ord}D_B-\mathrm{ord}\frkd_B)/2,
\]
where $\mathrm{ord}$ is the valuation of $F$.

We recall the theory of Siegel series (cf. Shimura \cite{shimura:94}, \cite{shimura:97}).
For $B\in\mathrm{Sym}_n(F)^{\mathrm{ss}}$, we define a polynomial 
$\gam(B, X)\in\ZZ[X]$ by
\[
\gam(B, X) 
=
\begin{cases}
(1-X)(1-q^{n/2}\xi_B X)^{-1}\prod_{i=1}^{n/2}(1-q^{2i}X^2) & \text{ if $n$ is even,} \\
\noalign{\vskip 5pt}
(1-X) \prod_{i=1}^{(n-1)/2}(1-q^{2i}X^2) & \text{ if $n$ is odd.}
\end{cases}
\]
Let $f_0^{(s)}$ be the function on the symplectic group $\SP_n(F)$ defined by 
\[
f_0^{(s)}(g)=|\det A|^{s+((n+1)/2)},
\]
for
\[
g=\begin{pmatrix} A & 0 \\ 0 & {}^t\!A^{-1} \end{pmatrix}\bfn(B)u,\qquad A\in\GL_n(F), B\in\mathrm{Sym}_n(F), u\in \SP_n(\frko).
\]
Then $f_0^{(s)}$ is a class one vector for the degenerate principal series $I(\mathbf{1}, s)=\mathrm{Ind}_{P_n}^{\SP_n}|\det|^{s}$.

Consider the integral
\[
b(B, s)=
|2|^{-n(n-1)/4}
\int_{N_{n}(k)} f_0^{(s-((n+1)/2))}\left(w_n \bfn(z)\right)\overline{\psi_B(z)}\, dz.
\]
This integral is absolutely convergent for $\mathrm{Re}(s)\gg 0$.
Moreover, there exists a polynomial $F(B, X)\in \ZZ[X]$ such that 
\[
b(B, s)=\gam(B, q^{-s})F(B; q^{-s}).
\]
For a proof of this fact, see \cite{shimura:97}.
Let 
\[
\calh_n(\frko)=\{ B=(b_{ij})\in \mathrm{Sym}_n(F)\,|\, b_{ij}\in 2^{-1}\frko, \, b_{ii}\in \frko\; (1\leq i\leq j \leq n)\}
\]
be the set of half-integral symmetric matrices of $F$.
It is known  that $F(B, X)=0$ unless $B\in \calh_n(\frko)$.
Moreover, if $B\in\calh_n(\frko)$, then $F(B, 0)=1$.

\begin{theorem}\label{prop:4.1}
The following functional equations hold.
\begin{enumerate}
\item If $n$ is even, then
\[
F(B, q^{-n-1}X^{-1})=(q^{(n+1)/2}X)^{-2\del_B} F(B, X).
\]
\item If $n$ is odd, then 
\[
F(B, q^{-n-1}X^{-1})=\eta_B (q^{(n+1)/2}X)^{-\mathrm{ord}(D_B)} F(B, X).
\]
\end{enumerate}
\end{theorem}
This theorem was first proved by Katsurada \cite{katsurada} for $F=\QQ_p$.
See also \cite{boechererkohnen} and \cite{satohironaka}.
\begin{proof}
To simplify the notation, we put 
\begin{align*}
\zeta(s)=&L(s, \mathbf{1})=(1-q^{-s})^{-1}, \\
L(s)=&L(s, \chi_{D_B})=(1-\xi_B q^{-s})^{-1}.
\end{align*}
It is well-known that
\[
M(s)f_0^{(s)}=
|2|^{n(n-1)/4}
\frac{\zeta(s-\frac{n-1}2)}{\zeta(s+\frac{n+1}2)}
\prod_{i=1}^{[n/2]} \frac{\zeta(2s-n+2i)}{\zeta(2s+n+1-2i)} f_0^{(-s)}.
\]
By Lemma \ref{lem:4.3}, we have
\begin{align*}
&
\gam(B, q^{s-((n+1)/2)}) F(B; q^{s-((n+1)/2)})
\\
&\quad \times |2|^{n(n-1)/4} \frac{\zeta(s-\frac{n-1}2)}{\zeta(s+\frac{n+1}2)} 
\prod_{i=1}^{[n/2]} \frac{\zeta(2s-n+2i)}{\zeta(2s+n+1-2i)} \\
=&
c_B(\mathbf{1}, s) |\det B|^{-s} 
\gam(B, q^{-s-((n+1)/2)}) F(B; q^{-s-((n+1)/2)}).
\end{align*}
We first prove (1).
If $n$ is even, then we have
\begin{align*}
c_B(\mathbf{1}, s)=
&|2|^{-ns+(n(n-1)/4)} |\frkd_B|^s 
\frac{L(-s+\frac12)}{L(s+\frac12)}
 \\
&\quad \times
\frac{\zeta(s-\frac{n-1}2)}{\zeta(-s+\frac{n+1}2)}
\prod_{i=1}^{n/2} 
\frac{\zeta(2s-n+2i)}{\zeta(-2s+n+1-2i)}, \\
\gam(B, q^{s-((n+1)/2)})=&
\frac{L(-s+\frac12)} {\zeta(-s+\frac{n+1}2)}
\prod_{i=1}^{n/2} \frac{1}{\zeta(-2s+n+1-2i)}, \\
\gam(B, q^{-s-((n+1)/2)}) =&
\frac{L(s+\frac12)}{\zeta(s+\frac{n+1}2)}
\prod_{i=1}^{n/2} \frac{1}{\zeta(2s+n+1-2i)}.
\end{align*}
Hence we have (1).
If $n$ is odd, then we have
\[
c_B(\mathbf{1}, s)=
|2|^{-(n-1)ns+(n(n-1)/4)}
\frac{\zeta(s-\frac{n-1}2)}{\zeta(-s+\frac{n+1}2)}
\prod_{i=1}^{n/2} 
\frac{\zeta(2s-n+2i)}{\zeta(-2s+n+1-2i)}, 
\]
\begin{align*}
\gam(B, q^{s-((n+1)/2)})=&\zeta(-s+\frac{n+1}2)^{-1}
\prod_{i=1}^{(n-1)/2} \zeta(-2s+n+1-2i)^{-1}, \\
\gam(B, q^{-s-((n+1)/2)}) =&\zeta(s+\frac{n+1}2)^{-1}
\prod_{i=1}^{(n-1)/2} \zeta(2s+n+1-2i)^{-1}.
\end{align*}
\end{proof}

\end{document}